\documentclass[11pt,a4paper]{article}
\usepackage[textwidth=16cm, textheight=23cm,marginratio=1:1]{geometry}
\usepackage{tikz-cd}
\usepackage[utf8]{inputenc}
\usepackage[OT1]{fontenc}
\usepackage[english]{babel}
\usepackage{dsfont}
\usepackage{amsfonts}
\usepackage{amsmath}
\usepackage{wasysym}
\usepackage{amsthm}
\usepackage{amssymb}
\usepackage{amsthm}
\usepackage{float}
\usepackage{textcomp}
\usepackage{xcolor}
\usepackage{graphicx}
\usepackage{wrapfig}
\usepackage{tikz}
\usepackage{mathdots}
\usepackage{fancybox}
\usepackage{fancyhdr}
\usepackage{mathrsfs}
\usepackage{latexsym, graphics,graphicx}

\usepackage[all]{xy}
\usepackage{centernot}
\usepackage{listings}
\usepackage{faktor}
\usepackage{bbm}
\usepackage{bm}
\usepackage{hyperref}
\usepackage{newlfont}
\usepackage{geometry}
\usepackage{ccaption}
\usetikzlibrary{matrix,arrows,decorations.pathmorphing}
\usepackage{guit}
\usepackage{calligra,amsmath,amssymb}
\usepackage{calc}
\usepackage{lmodern}

\newtheorem{proposition}{Proposition}[section]
\newtheorem*{proposition*}{Proposition}
\newtheorem{lemma}[proposition]{Lemma}
\newtheorem*{lemma*}{Lemma}
\newtheorem{corollary}[proposition]{Corollary}
\newtheorem*{corollary*}{Corollary}
\newtheorem{theorem}[proposition]{Theorem}
\newtheorem*{theorem*}{Theorem}

\theoremstyle{definition}

\newtheorem{remark}[proposition]{Remark}
\newtheorem*{remark*}{Remark}
\newtheorem{example}[proposition]{Example}
\newtheorem*{example*}{Example}

\newcommand{\stirlingii}{\genfrac{\{}{\}}{0pt}{}}
\newcommand{\cA}{\mathcal{A}}
\newcommand{\cB}{\mathcal{B}}
\newcommand{\CC}{\mathbb{C}}
\newcommand{\cO}{\mathcal{O}}
\newcommand{\OO}{\mathcal{O}}
\newcommand{\PP}{\mathbb{P}}

\DeclareMathOperator{\HH}{H}

\DeclareMathOperator{\rk}{rk}

\DeclareMathOperator{\Div}{Div}

\DeclareMathOperator{\Td}{Td}
\usepackage{hyperref} 
\hypersetup{colorlinks=true,linkcolor=red}

\numberwithin{equation}{section}

\title{Equivariant Euler characteristics on permutohedral varieties}
\author{Vincenzo Galgano\footnote{Dipartimento di Matematica, Università di Trento, Via Sommarive 14 38123 Trento, Italy; ORCID: 0000-0001-8778-575X: \texttt{vincenzo.galgano@unitn.it}} , Hanieh Keneshlou\footnote{{Department of Mathematics and Statistics, Universitätsstraße 10, 78464 Konstanz, Germany: \texttt{hanieh.keneshlou@uni-konstanz.de} }} , Mateusz Michałek\footnote{{Department of Mathematics and Statistics, Universitätsstraße 10, 78464 Konstanz, Germany: \texttt{mateusz.michalek@uni-konstanz.de}}}}
\date{\today}

\begin{document}

\maketitle

\begin{abstract}
\noindent By the work of J.~Huh, one can interpret binomial coefficients as a solution to an intersection problem on a permutohedral variety $X_E$. Applying Hirzebruch-Riemann-Roch, this intersection problem is equivalent to computing Euler characteristic of a specific element of $K$-theory of $X_E$. This element has a natural lifting to equivariant $K$-theory and thus the Euler characteristic may be upgraded to a Laurent polynomial. We provide and implement three different approaches, in particular a recursive one, to computing these polynomials.
\hfill\break
\end{abstract}

\tableofcontents

\section*{Acknowledgement}
VG has been supported by the italian “National Group for Algebraic and Geometric Structures, and their Applications” (GNSAGA-INdAM). MM and HK are funded by the Deutsche Forschungsgemeinschaft — Projektnummer 467575307. 
We would like to thank Andrzej Weber for many useful discussions. We thank Julian Weigert for implementing the recursive formula, as well as observing identities among involved Stirling numbers.

\section{Introduction}
\indent \indent Binomial coefficients have played a very important role in mathematics for over thousand years. Indeed, the first known written exposition of what we now call Pascal triangle comes from an Indian mathematician Halayudha from 10-th century. In this article, we study naturally defined multivariate Laurent polynomials which are extensions of binomial coefficients: evaluated at $(1,\dots,1)$ they give binomial coefficients. \\
\indent The motivation for our work comes from a modern, groundbreaking work of June Huh \cite{huh2012milnor}, who related basic combinatorial invariants of graphs and matroids with multidegree of complex projective algebraic varieties, and more precisely with intersection theory on a permutohedral variety. To any $\CC$-representable matroid $M$ on a set $E$ of cardinality $n+1$ one associates a class $Y_M$ in the Chow ring of an $n$-dimensional permutohedral variety $X_E$, the latter being the blow-up of the projective space $\PP^n$ at its torus invariant linear subspaces, or equivalently the resolution of the projective space $\mathbb P^n$ at the Cremona singularities. In particular, $X_E$ comes with two projections to $\PP^n$ and we consider the divisors $\Gamma$ and $\Delta$ given by the pullback of a hyperplane under the respective projections. If we fix a representation of $M$, we may consider $Y_M$ as an irreducible subvariety of $X_E$. By \cite{huh2012milnor, adiprasito2018hodge} the absolute value of the $a$-th coefficient of the reduced chromatic polynomial $\chi_M$ of $M$ is equal to $\int_{X_E} Y_M [\Gamma]^a [\Delta]^{r-a}$, where the product is taken in the Chow ring $A(X_E)$ of $X_E$ and $r$ is the rank of $M$.\\ 
\indent In fact, in \cite{adiprasito2018hodge} the authors associate to any matroid a ring, called the Chow ring, and prove (analogue of) Hard Lefschetz theorem and Hodge-Riemann relations for it. This major breakthrough allowed, among others, to prove log-concavity conjectures for coefficients of the characteristic polynomial of a matroid. This was a generalization of the case of matroids representable over $\CC$, where the Chow ring was indeed a Chow ring of a complex, compact variety and Lefschetz theorem was known \cite{huh2012milnor}. In this article we focus our attention on other properties of this ring, looking at intersections with pull-backs of hyperplane divisors from the equivariant perspective and through the isomorphism of the Chow ring with the $K^0$-theory ring given by the Chern character.\\
\indent Arguably, the simplest case of this theory is when $M$ is a uniform matroid of rank $r$. In this case, the absolute value of the $i$-th coefficient (from the top) of the reduced chromatic polynomial equals $\binom{n}{i}$. Further, the class $Y_M$ is simply given as a power of $[\Gamma]$. We thus obtain
\[ \int_{X_E}[\Gamma]^a\cdot [\Delta]^{n-a}={n\choose a} \ . \]
\indent Starting from the above equality, our aim is to take one step further and take into account the natural torus action on $X_E$. The key for this is the passage from the Chow-theoretic expression of this intersection to the $K$-theoretic expression of it in the Grothendieck ring $K^0(X_E)$ via the Hirzebruch-Riemann-Roch formula $\chi(\cal F)= \int_{X_E} ch(\cal F)\Td(X_E)$. This leads in our case to
\begin{equation}
\chi\left( ([\OO_{X_E}]-[\OO_{X_E}(-\Gamma)])^{\otimes a}\otimes ([\OO_{X_E}]-[\OO_{X_E}(-\Delta)])^{\otimes n-a} \right) = \int_{X_E}[\Gamma]^a\cdot [\Delta]^{n-a} \ .  
\end{equation}
We note that all of the line bundles on the left-hand side come with a natural torus action. 
Thus our computation may be upgraded to the computation of the equivariant Euler characteristic of a certain class in the equivariant Grothendieck ring $K^{0}_T(X_E)$. This is now a symmetric Laurent polynomial that evaluates to the given binomial coefficient on $(1,\dots, 1)$. The polynomials we obtain are highly non-trivial and we provide three different ways to compute them. The challenge to compute the given Euler characteristic lead us, among others, to proving various vanishing theorems, exact formulas for equivariant cohomology groups for various bundles on the permutohedral variety as well as exhibiting relations to other mathematical object, like Stirling numbers of the second kind. We find these results interesting on their own and plan to exploit them further in future. \\
\hfill\break
\indent This paper is structured as follows. In Section \ref{prem} we start by recalling the construction of the permutohedral variety and the description of important divisor classes on it. We also briefly recall the equivariant localization formula and discuss in detail the case of the projective line. In Section \ref{Comput}, we start by proving the vanishings of the intermediate cohomology groups for various sheaves, which simplifies the expression of the Euler characteristic. Here our main tools are toric Kawamata-Viehweg vanishing and Batyrev-Borisov vanishing.  Combined with equivariant Serre duality, we compute the equivariant Euler characteristics as integral combinations of global section of toric line bundles, cf.~formula \eqref{chi as global sections}. This results in a direct, albeit long, formula where the final Laurent polynomial is a result of identifying lattice points with Laurent monomials and summing over various polytopes. It could be seen as an analogue of the formula \cite[Theorem 3.11]{mt}, where the variety of complete quadrics is considered, while we deal with the permutohedral variety. Our second approach is based on using the so-called equivariant localization formula. Here the final result is presented as a sum of rational functions, cf.~formula \eqref{chi(F)withSigma}. In Section \ref{Recursive} we provide a recursive formula that reduces the computation of the Euler characteristic to the computation of classes on lower dimensional permutohedra. It could be regarded as an equivariant analog of the formulas in the proof of \cite[Lemma I.1]{tautologicalclasses}. However, there is one central difference: in non-equivariant Chow ring intersecting more divisors than the dimension of the variety always gives zero, and thus such contributions disappear; in the equivariant Chow ring, we obtain several contributions related to intersecting more divisors than the dimension of the variety.\\ 
\hfill\break
\indent The three presented approaches are implemented in the computer algebra system \textit{Macaulay2} \cite{M2} and they are available under \texttt{EquivariantEC.m2} as supplementary material to this article.
Finally, we would like to point out that our approach to upgrading important numeric invariants to polynomials is certainly not the only possibility. In particular, in the recent work \cite{tautologicalclasses} there is an explicit formula for $Y_M$ in the equivariant Chow ring. We leave the implementation and comparison of those for future work.

\section{Preliminaries}\label{prem}

\subsection{Permutohedral varieties and special divisors}

\paragraph{Permutohedral varieties.} We briefly recall the construction and basic properties of the permutohedral variety: for more details we refer to \cite{huh}. Our main reference for the general theory of toric varieties is \cite{cox} but we also suggest \cite[Chapter 8]{michalek2021invitation} for a fast introduction to it. \\
\indent The 
{\em permutohedron} $\Pi_{n+1}\subset \mathbb R^{n+1}$ is the convex hull of the points $(\sigma(1),\ldots, \sigma(n+1))\in \mathbb R^{n+1}$ as $\sigma \in \mathfrak S_{n+1}$ varies among the permutations of $n+1$ elements. Let $T=(\mathbb C^{\times})^{n+1}$ be the algebraic torus with the character lattice given by $\mathbb Z^{n+1}\subset \mathbb R^{n+1}$. The $n$-dimensional {\em permutohedral variety} $X_n$ is the projective toric variety associated to the polytope $\Pi_{n+1}$. 
We note that there are two tori acting on $X_{n} $, one of which is the torus $T$. For any $\lambda\neq 0$ the element $(\lambda,\dots,\lambda)\in T$ stabilizes every point $x\in X_n$. Thus the second torus is $T':=T/(\lambda,\dots,\lambda)_{\lambda\in\mathbb C^\times}$ having character lattice $M_{T'}:=\{(m_1,\dots,m_{n+1})\in \mathbb Z^{n+1} \ | \ \sum m_i=0\}$.

 The rays in the fan $\Sigma_{X_n}$ of $X_n$ are parametrized by the non-empty proper subsets of $[n+1]:=\{1, \ldots, n+1\}$. Given $(e_1,\ldots, e_{n+1})$ the standard basis of $\mathbb R^{n+1}$, they are the following $2^{n+1}-2$ vectors in the lattice $M^*=\mathbb Z^{n+1}/\langle \mathbbm{1} \rangle$ of one-parameter subgroups of $T'$, where $\mathbbm{1}:=(1,\ldots, 1)$:
\begin{equation}\label{rays}
\forall S \subsetneq [n+1], S\neq \emptyset , \ \ \ u_S:=\sum_{i\in S}e_i \ \ (mod \  \mathbbm{1}) \ .
\end{equation}

\noindent Notice that the following relations hold {\em modulo} $\mathbbm{1}$:
\begin{equation}\label{rays relations}
u_S= -u_{[n+1]\setminus S}, \ \ \ \sum_{i \in S, j \notin S} u_S= -\sum_{i \notin S, j \in S} u_S \ \ \ \forall i,j\in[n+1] \ . 
\end{equation}
Considering the natural action of $T$ on the projective space $\PP^n$, the variety $X_n$ can also be defined as the consecutive blow-up of $\mathbb P^n$ at all the $T$-invariant linear subspaces, starting from points, through lines and ending with codimension two coordinate subspaces. This blow-up construction resolves the classical birational Cremona map:
\begin{center}
	\begin{tikzpicture}[scale=2.5]
	\node(X) at (0,0.6){$X_n$};
	\node(PP) at (-0.6,0){$\mathbb P^{n}$};
	\node(PPvee) at (0.6,0){$(\mathbb P^{n})^\vee$};
	\node(pp) at (-0.6,-0.25){$[x_0: \ldots : x_n]$};
	\node(pp2) at (0.6,-0.25){$[x_0^{-1}: \ldots : x_{n}^{-1}]$};	
	
	\path[font=\scriptsize,>= angle 90]
	(PP) edge [dashed,->] node [above] {$Crem$} (PPvee)
	(X) edge [->] node [left] {$p$} (PP)
	(pp) edge [|->] node [] {} (pp2)
	(X) edge [->] node [right] {$q$} (PPvee);	
	\end{tikzpicture}
\end{center}
\paragraph{Divisors on the permutohedral variety.} As for any toric variety, the group of $T$-invariant divisors $\Div^T(X_n)$ is generated by the so-called {\em prime} $T$-invariant divisors corresponding to the rays of $\Sigma_{X_n}$. For every non-empty proper subset $S\subsetneq [n+1]$, let $D_S\in \Div^T(X_n)$ be the  prime $T$-invariant divisor corresponding to the ray $u_S$ so that any $T$-invariant divisor $D \in \Div^T(X_n)$ is of the form $D= \sum_{\emptyset\neq S\subsetneq [n+1]} a_SD_S$ for certain $a_S \in \mathbb Z$. For simplicity, from now on we always write $S\subset [n+1]$ meaning only non-empty proper subsets.\\
Let $H_i\subset \PP^n$ and $\hat H_j\subset (\mathbb P^n)^\vee$ be the two coordinate hyperplanes corresponding to the $i$-th and the $j$-th coordinates, respectively. Then, one has the expression of the pullback of the hyperplanes under the respective projection map (see \cite[Subsection 3.1]{huh}):
\begin{equation}\label{Gamma_i, Delta_j}
\Gamma_i := p^\ast (H_i) = \sum_{i \in S}D_S, \ \ \ \Delta_j := q^\ast (\hat H_j) = \sum_{j\notin S}D_S
\end{equation}
and the canonical divisor on $X_n$ is 
\begin{equation}\label{canonical divisor} 
K:= K_{X_n} \ = \ - \sum_{S\subset [n+1]} D_S \ = \  - \Gamma_i - \Delta_i \ .\end{equation}

For our purpose of proving the vanishing of the intermediate cohomology groups, we show the following result.
\begin{lemma}\label{lem:bigandnef}
For each $k,l\geq 0$, not both equal to zero, the divisor $k\Gamma_i + \ell \Delta_j\in \Div^T(X_n)$ is nef and big. In particular, the anticanonical divisor $-K$ is big and nef. 
\end{lemma}
\begin{proof}
  As the pullback of a very ample divisor under a generically finite map is big and nef (see \cite[Subsection 2.2.A]{lazarsfeld}), the divisors $\Gamma_i$ and $\Delta_j$ are big and nef. In particular, for any $k,l\geq 0$ (not both zero) the divisor $k\Gamma_i + \ell \Delta_j\in \Div^T(X_n)$ is big and nef. Thus, the anticanonical divisor $-K=\Gamma_i + \Delta_i$ is big and nef.  
\end{proof}

\subsection{Equivariant localization formula}

\indent \indent There is a bijection among permutations $\sigma \in \mathfrak S_{n+1}$, maximal cones $C_\sigma \in \Sigma_{X_n}(n)$ in the fan $\Sigma_{X_n}$ and torus-fixed points $P_\sigma \in X^T_n$ given by
\begin{equation}\label{bij S_n X^T}
	\begin{matrix}
		\mathfrak S_{n+1} & \longleftrightarrow & \Sigma_X(n) & \longleftrightarrow & X_n^T\\
		\sigma & \leftrightarrow & C_\sigma \ = \ Cone\!\left(u_{S_1}, \ldots, u_{S_n}\right) & \leftrightarrow & P_\sigma
	\end{matrix}
	\end{equation}
where $S_i=\{\sigma(1), \ldots, \sigma(i)\}$ for any $1\leq i \leq n$. For any $\sigma \in \mathfrak S_{n+1}$, there is an action of the torus $T$ on the affine chart $U_\sigma \ni P_\sigma$, thus an action of $T$ on the tangent space $T_{P_\sigma}(X_n)$. To simplify the notation, we write $T_\sigma(X_n)$ instead of $T_{P_\sigma}(X_n)$. The following result is well-known to experts.

\begin{proposition}
	For any $\sigma \in \mathfrak S_{n+1}$, taking into account the torus action, the tangent space to $X_n$ at $P_\sigma$ is
	\begin{equation}\label{tg generators}
		T_\sigma (X_n) = \left\langle t_{\sigma(1)}t_{\sigma(2)}^{-1}, \ldots, t_{\sigma(n)}t_{\sigma(n+1)}^{-1}\right\rangle_\mathbb C \ .
	\end{equation}
	\end{proposition}

Let $K^0_T(X_n)$ be the $T$-equivariant Grothendieck ring of (the equivalent classes of) vector bundles on $X_n$, and let $K^0_T(X_n^T)$ be the one for $X_n^T$. For any $[\cal E]$ in $K^0_T(X_n)$, one denotes its dual class by $[\cal E^\vee]$. It is known that for a point $P$ with the trivial $T$-action, there exists an identification
\[ K^0_T(P) \sim \mathbb Z[t_1^{\pm 1}, \ldots , t_{n+1}^{\pm 1}],\]
associating to any class of vector bundle on $P$ (i.e.~a $T$-representation) the sum of the characters giving its decomposition into irreducible $T$-representations. In particular, it easily follows that we may identify
\[ K^0_T(X_n^T) \sim \prod_{\sigma \in \mathfrak S_{n+1}} \mathbb Z[t_1^{\pm 1}, \ldots, t_{n+1}^{\pm 1}] \ .\]
For any $f \in \prod_{\sigma \in \mathfrak S_{n+1}} \mathbb Z[t_1^{\pm 1}, \ldots, t_{n+1}^{\pm 1}]$ we denote by $f_\sigma$ its projection onto the $\sigma$-th factor.

\begin{theorem}[\cite{tautologicalclasses}, Theorem 2.1]
	The restriction map $K^0_T(X_n)\rightarrow K^0_T(X_n^T)$ is injective and its image is identified with the subring
	{\footnotesize \[ \left\{ f \in \prod_{\sigma \in \mathfrak S_{n+1}} \mathbb Z[t_1^{\pm 1}, \ldots, t_{n+1}^{\pm 1}] \ \bigg| \ f_\sigma \equiv f_{\tau} \left(\text{mod} \ \big(1-\frac{t_{\sigma(i+1)}}{t_{\sigma(i)}}\big)\right) \ , \forall \tau=\sigma \circ (i,i+1) \right\} \ .\]}
	\end{theorem}
The above result allows us to denote by $[\cal E]_\sigma\in \mathbb Z[t_1^{\pm 1},\ldots, t_{n+1}^{\pm 1}]$ the restriction of $[\cal E]\in K^0_T(X_n)$ to the $T$-fixed point $P_\sigma \in X_n^T$. Given $m=(m_1,\ldots,m_{n+1})\in \mathbb Z^{n+1}$, we write $\mathbf{t}^m=t_1^{m_1}\cdots t_{n+1}^{m_{n+1}}$. Then, for every $[\cal E]\in K^0_T(X_n)$ and for every $\sigma \in \mathfrak S_{n+1}$ it holds
\begin{equation}\label{fiber of vector bundle as characters} 
	[\cal E]_\sigma= \sum_{l=0}^{k_\sigma}a_{\sigma,l}\mathbf{t}^{m_{\sigma ,l}}, \qquad [\cal E^\vee]_\sigma= \sum_{l=0}^{k_\sigma}a_{\sigma, l}\mathbf{t}^{-m_{\sigma, l}}
	\end{equation}
for certain $k_\sigma \geq 0$, $a_{\sigma, l}\in \{\pm 1\}$ and $m_{\sigma,l}\in \mathbb Z^{n+1}$ so that $\rk(\cal E)=a_{\sigma,1}+\ldots + a_{\sigma, l}$ (see \cite[Subsection 2.5]{tautologicalclasses} for details). \\

The Cremona map on $\mathbb P^n$ lifts to an involution map on $X_n$, which we also denote by $crem$, commuting the following diagram:
\begin{equation}\label{cremona diagram}
	\begin{tikzpicture}[scale=2.5]
	\node(a) at (-0.5,0.3){$X_n$};
	\node(b) at (0.5,0.3){$X_n$};
	\node(c) at (-0.5,-0.3){$\mathbb P^n$};
	\node(d) at (0.5,-0.3){$\mathbb P^n$};
	
	\path[font=\scriptsize,>= angle 90]
	(a) edge [->] node [above] {$crem$} (b)
	(a) edge [->] node [left] {$p$} (c)
	(c) edge [->,dashed] node [below] {$crem$} (d)
	(b) edge [->] node [right] {$p$} (d)
	(a) edge [->] node [above] {$q$} (d);
	\end{tikzpicture}
	\end{equation}
The map $crem: X_n\rightarrow X_n$ is not $T$-invariant, but it is a toric morphism, thus it preserves $X_n^T$. More precisely, for any $\sigma \in \mathfrak S_{n+1}$ it holds $crem(P_\sigma)=P_{\overline\sigma}$, where $\overline \sigma(i):= \sigma(n+1-i)$ for all $1 \leq i\leq n+1$. Moreover, the Cremona map induces an involution $crem: K^0_T(X_n)\rightarrow K^0_T(X_n)$ under which one has:
\begin{equation}\label{fiber via crem} 
	[\cal E]_\sigma=\sum_{l=1}^{k_\sigma}a_{\sigma,l}\mathbf{t}^{m_{\sigma , l}} \ \ \implies \ \ \left(crem[\cal E]\right)_\sigma =  \sum_{l=1}^{k_{\overline\sigma}}a_{\overline \sigma, l}\mathbf{t}^{-m_{\overline \sigma, l}} \ .
	\end{equation}

The following result is known as the {\em equivariant localization formula}.
\begin{theorem}[\cite{localization}, 4.7]\label{formula}
 For any $[\cal F]\in K^0_T(X_n)$ it holds
 \begin{equation}\label{localization formula}
  \chi_{_T}([\cal F])= \sum_{\sigma\in \mathfrak S_{n+1}} \dfrac{[\cal F]_\sigma}{\prod_{s=1}^{n}\left(1-t_{\sigma(s)}^{-1}t_{\sigma(s+1)}\right)} \ .
 \end{equation}   
\end{theorem}

\subsection{Example of projective line}\label{subsec:example P1}

\indent \indent In this section, we present our approach in the simplest case when $n=1$ and $X_1=\PP^1$. In general, we will not be working over a projective space, thus
the computations will be much more involved. However, we hope that the example below can give the flavor of our results. 
In particular, our aim here is to highlight the differences among:
\begin{itemize}
\item equivariant and non-equivariant case,
\item the big $T:=(\CC^\times)^2$ and small $T':=(\CC^\times)^2/(\lambda,\lambda)$ tori actions,
\item different methods of computation.
\end{itemize}

\indent Let us start with the non-equivariant case. Let $H$ be a point in $\PP^1$ and let $\cO(1)$ be the associated line bundle. As the Cremona map is a linear
isomorphism, $H$ plays the role of both divisors: $\Gamma$ and $\Delta$. By abuse of notation, let $H$ be also the corresponding class in the Chow ring $A(\PP^1)$.
Clearly, we have $H=1 pt$ and $H^a=0$ for $a>1$. The Hirzebruch-Rieman-Roch formula immediately implies:
\begin{align}\label{eq:noneq}
\chi([\cO]-[\cO(-1)])=1,\quad \chi(([\cO]-[\cO(-1)])^{\otimes a})=0, \ \ \ \forall a>0.\end{align}
Here the situation is exceptionally nice. By the Chern character isomorphism, $[\cO]-[\cO(-1)]$ is mapped exactly to $H$. For larger $n$, we also obtain higher-order
terms that are not visible in such a small example. We note that the above formulas may also be computed directly, as all cohomology groups of line bundles
on projective spaces are known \cite[Theorem III.5.1]{Hartshorne}.\\

We now turn to the equivariant case. Both $T$ and $T'$ act on $X_1=\PP^1$. From the point of view of toric geometry, it is more natural to choose the action of $T'$.
However, once we consider the line bundle $\cO(-1)$ the only natural action is that of $T$. Here, we have two choices: we either have to
work with $T'$ and choose a linearization of the line bundle or we work with $T$, but we remember that our group has a stabilizer. In the article, in general,
we cover both of those approaches. \\

Let us start with $T'$. We have two torus invariant divisors $H_1:=[1:0]$ and $H_2:=[0:1]$. These two correspond to choosing the action of $T'$ on $\cO(1)$ or 
equivalently on $\cO(-1)$. In fact, there are more choices, as we could represent $H=2H_2-H_1$, however, we focus on those two. Hence, we fix $H_i$, where
$i=1$ or $i=2$. We may now compute all cohomology groups as a representation of $T'$. That is, we identify each cohomology group, and hence Euler characteristics, 
with the representation of $T'$ and consequently with the character, that is a Laurent polynomial. 
This approach is carried out in general in Section \ref{sec:firstap}. First, we obtain:
\[\chi_{_{T'}}\big([\cO]-[\cO(-H_i)]\big)=\chi_{_{T'}}([\cO])=1.\]
This looks similar to the formula \eqref{eq:noneq}, however now the value $1$ on the right-hand side is a polynomial in the variables $t_1, t_2$ that evaluated at the point $(1,1)$ gives $1$ from 
\eqref{eq:noneq}. A more interesting case is $a=2$:
\begin{align*} 
    \chi_{_{T'}}\big(([\cO]-[\cO(-H_i)])^{\otimes 2}\big) & =\chi_{_{T'}}([\cO])-2\chi_{_{T'}}([\cO(-H_i]))+\chi_{_{T'}}([\cO(-2H_i)])\\
    &=\chi_{_{T'}}([\cO])+\chi_{_{T'}}([\cO(-2H_i)])
    =1-\HH^1(\mathbb P^1,\cO(-2H_i)) \ .
    \end{align*}
Here, we used the fact that many cohomology groups vanish, which in this case can be deduced from \cite[Theorem III.5.1]{Hartshorne}. However, as we will see
in the article, more powerful tools are needed in general. In order to compute the top cohomology group, we apply toric Serre duality \cite[Theorem 9.2.10]{cox}, which
gives us:
\[\HH^1(\mathbb P^1,\cO(-2H_i))=\HH^0(\mathbb P^1,\cO(-H_1-H_2+2H_i))^\vee=\frac{t_1t_2}{t_i^2}.\]
Hence, we obtain
\[ \chi_{_{T'}}\big(([\cO]-[\cO(-H_i)])^{\otimes 2}\big)=1-\frac{t_1t_2}{t_i^2},\]
which indeed evaluated at $(1,1)$ gives $0$, but clearly it is a non-zero Laurent polynomial. Alternatively, if we replace $([\cO]-[\cO(-H_i)])^{\otimes 2}$ by $([\cO]-[\cO(-H_1)])
\otimes([\cO]-[\cO(-H_2)])$, we get the equivariant Euler characteristic as
\[ {\chi}_{_{T'}}\big( ([\cO]-[\cO(-H_1)])
\otimes([\cO]-[\cO(-H_2)]) \big) =    {\chi}_{_{T'}}( [\cO] ) + {\chi}_{_{T'}}([\cO(-H_1-H_2)]) = 1 - \HH^0(\mathbb P^1,\cO)^\vee = 0, \]
which actually is the zero polynomial (not only its evaluation at $(1,1)$ is zero). 

Next, we consider the action of $T$. While it is still possible to perform the toric computations as above, let us present a different method, based on equivariant localization.
In general, this is carried out in Section \ref{sec:eqloc}. Note that $\cO(1)$ comes with the natural $T$-action. The two global sections $x_1$ and $x_2$ are acted on as follows:
\[(t_1,t_2) \cdot x_1=t_1^{-1}x_1, \quad (t_1,t_2)\cdot  x_2=t_2^{-1} x_2.\]
The inverses come from the fact that $T$ acts on $\PP^1=\PP(\CC^2)$ and the global sections of $\cO(1)$ are identified naturally with $(\CC^2)^\vee$. The very powerful method to deal
with equivariant $K$-theory is, for any vector bundle, to look at the vector spaces over torus invariant points. For $\cO(-1)$ we obtain the following two vector spaces:
\begin{itemize}
\item $(\star,0)\subset \CC^2$ over $[1:0]$ and
\item $(0,\star)\subset \CC^2$ over $[0:1]$. 
\end{itemize}
Here, the reader may recall that $\cO(-1)$ is the universal bundle over the projective space. In particular, we obtain:
\begin{itemize}
\item the character $t_1$ over $[1:0]$ and
\item the character $t_2$ over $[0:1]$.
\end{itemize}
To simplify notation, we will write the pair of characters $(t_1,t_2)$ to denote the class $[\cO(1)]$ over the two torus fixed points. In fact, previously, the linearizations of $H_1$ and $H_2$,
with respect to the torus $T'$, correspond to $(t_1/t_2,1)$ and $(1, t_2/t_1)$, where now the Laurent monomials are not only characters of $T$ but also of $T'$.
However, for now, we fix $(t_1,t_2)$, so that no particular choices need to be made. Furthermore, the class $[\cO(-2)]$ is represented by $(t_1^2, t_2^2)$, as the tensor product of classes corresponds to the multiplication of characters.
We also represent $[\cO]-[\cO(-1)]$ by $(1-t_1, 1-t_2)$. Once we know the cotangent directions at every point, we may compute the Euler characteristic using the formula from Theorem \ref{formula}.
In this example, for $[1:0]$ we have one character $t_1/t_2$ and for $[0:1]$ we have $t_2/t_1$. Hence, we have:
\[\chi_{_T}([\cO]-[\cO(-1)])=\frac{1-t_1}{1-t_1/t_2}+\frac{1-t_2}{1-t_2/t_1}=1\]
as computed before so that the choice of linearization of $\cO(-1)$ is not seen here. However, as the $K$-class of $([\cO]-[\cO(-1)])^{\otimes 2}$ is represented by 
$((1-t_1)^2,(1-t_2)^2)$, we obtain
\[\chi_{_T}(([\cO]-[\cO(-1)])^{\otimes 2})=\frac{(1-t_1)^2}{1-t_1/t_2}+\frac{(1-t_2)^2}{1-t_2/t_1}=1-t_1t_2,\]
which compare to the previous computation, explains the difference in the choice of linearization of $\cO(-1)$.

\section{Computation of the Euler characteristic}\label{Comput}

\indent \indent In this section, we compute the equivariant Euler characteristic of the desired $K$-class in two ways.
Fix $i,j\in [n+1]$, it is known that for any $0\leq a \leq n$ the  following holds
\[  \int_{X_E}[\Gamma_i]^a\cdot [\Delta_j]^{n-a}=\binom{n}{a} \ .\]
However, the above equality is only a degree information, we are interested in determining the intersection class by taking into account the $T$-action. As already pointed out in Subsection \ref{subsec:example P1}, fixing the indices $i,j$ means choosing linearizations for $\cO(1)$. In the following, we mainly consider the action of the small torus $T'=T/\langle \mathbbm{1}\rangle$ and representation of the characters in its lattice $M_{T'}$.\\

In order to reach our goal, we pass from the Chow-theoretic expression of this intersection to the $K$-theoretic expression of it in the Grothendieck ring $K^{0}(X_n)$ via the Hirzebruch-Riemann-Roch formula $\chi(\cal F)= \int_{X_n} ch (\cal F)\Td(X_n)$. By \cite[Lemma~3.10]{mt}, we obtain
\[ \chi\left( \big([\OO_{X_n}]-[\OO_{X_n}(-\Gamma_i)]\big)^{\otimes a}\otimes \big([\OO_{X_n}]-[\OO_{X_n}(-\Delta_j)]\big)^{\otimes n-a} \right) =  \int_{X_E}[\Gamma_i]^a\cdot [\Delta_j]^{n-a} \ .\]
Our main aim is the computation of the equivariant Euler characteristic on the left-hand side
for the $T$-equivariant class \[\big([\OO_{X_n}]-[\OO_{X_n}(-\Gamma_i)]\big)^{\otimes a}\otimes \big([\OO_{X_n}]-[\OO_{X_n}(-\Delta_j)]\big)^{\otimes n-a} \in K^{0}_T(X_n).\]
To simplify the notation, we usually omit $X_n$ in the presentation of the $K$-classes of bundles in $K_T(X_n)$, and the tensor products in the exponents.\\
\indent Let $\cA_{n,i}:=([\OO]-[\OO(-\Gamma_{i})])$ and $\cB_{n,j}:=([\OO]-[\OO(-\Delta_{j})])$ be the $T$-equivariant classes on the $n$-dimensional permutohedral variety $X_n$, and let $\cA_{n,i}^a=\cA_{n,i}^{\otimes a}$ and $\cB_{n,j}^b=\cB_{n,j}^{\otimes b}$. In this notation, our goal can be reformulated in the computation of
$$\chi_{_T}(\cA_{n,i}^a\otimes \cB_{n,j}^{n-a} ). $$
By expanding the tensor products and by setting 
\[D_{ij}^{k\ell}:= k\Gamma_i + \ell \Delta_j \in \Div^T(X_n), \]
for any $0\leq k\leq a$ and $0\leq \ell\leq n-a$, we have
\[ \chi_{_T}(\cA_{n,i}^a\otimes \cB_{n,j}^{n-a} ) =  \sum_{k=0}^a\sum_{\ell=0}^{n-a}(-1)^{k+l}\binom{a}{k}\binom{n-a}{\ell}\chi_{_T}\left( [\OO(-D_{ij}^{k\ell})]\right). \]

In the following subsections, we compute the equivariant Euler characteristics $\chi_{_T}([\OO(-D_{ij}^{k\ell})])$ via two different approaches: the first approach consists in determining the cohomology groups of the line bundles $\OO(-D_{ij}^{k\ell})$, while the second one is based on the direct computation of $\chi_{_T}([\OO(-D_{ij}^{k\ell})])$ using the equivariant localization formula. Within the first approach, the equivariant Euler characteristic of a $K$-class is a formal sum of irreducible $T$-representations, which are always one-dimensional and correspond to characters on the torus, hence to points in $\mathbb Z^{n+1}$ or equivalently to Laurent polynomials in $\mathbb Z[t_1^{\pm 1}, \ldots, t_{n+1}^{\pm 1}]$, where a character $(m_1,\ldots, m_{n+1})$ is identified with the Laurent monomial $ t_1^{m_1}\cdots t_{n+1}^{m_{n+1}}$. Throughout this paper, we represent the Euler characteristic of a $K$-class as a Laurent polynomial.

\subsection{First approach: computing cohomology groups}\label{sec:firstap}
\indent \indent In this subsection, we determine the equivariant cohomology groups appearing in the equivariant Euler characteristic $\chi_{_T}\left( [\OO(-D_{ij}^{k\ell})]\right )$ as $k$ and $\ell$ vary.  First, we prove the vanishing of many intermediate cohomology groups and, consequently, we reduce the computation of the equivariant Euler characteristic to the computation of only the vector spaces of the global sections, that is $\HH^0$-terms. Since each such a cohomology group inherits the $T$-action on $X_n$, and hence it is a  representation of the torus $T$, we determine their irreducible components. \\

We recall the following vanishing results, used frequently in this section.
\begin{theorem}[Toric Kawamata-Viehweg; Theorem 9.3.10, \cite{cox}]\label{TKV}
	Let $X$ be a toric variety. For any big and nef divisor $D\in \Div^T(X)$, it holds
	\[ \HH^p(X,\OO(K_X+D))= 0, \ \ \ \forall p>0 \ .\]
	\end{theorem}

\begin{theorem}[Batyrev-Borisov vanishing; Theorem 9.2.7, \cite{cox}]\label{BB}
	Let $X$ be a toric variety. For any nef divisor $D \in \Div^T(X)$, it holds 
	\[ \HH^p(X,\OO(-D))=0, \ \ \ \forall p \neq \dim P_D \ .\]
 where $P_D$ is the associated polytope to the divisor $D$.
	\end{theorem}

\indent The first step is computing the equivariant Euler characteristic of the structure sheaf $\chi_{_T}([\OO])$. Since the anticanonical divisor $-K_{X}$ is big and nef, by Theorem \ref{TKV} we get the vanishings $\HH^p(X,\OO)=\HH^p(X,K_X+(-K_X))=0$ for any $p>0$ and we conclude that 
\begin{equation}\label{chi(O)}
	 \chi_{_T}([\OO])=\HH^0(X,\OO)=1. 
	\end{equation}
\indent Next, we consider the case in which either $k=0$ or $\ell=0$.
\begin{proposition}\label{k0ORl0}
    For any $1 \leq k, \ell\leq n$, it holds
    \[ \chi_{_T}([\OO(-k\Gamma_i)])=\chi_{_T}([\OO(-\ell\Delta_j)])=0 \ . \]
\end{proposition}
\begin{proof}
    By symmetry and freedom on the choice of the pullbacks $\Gamma_i$ and $\Delta_j$, the sheaves $\OO(-k\Gamma_i)$ and $\OO(-\ell \Delta_j)$ have common vanishings in cohomology, although their non-zero cohomology groups could split in different representations. In light of this, it is enough to prove the statement for $\chi_{_T}([\OO(-k\Gamma_i)])$ as $1\leq k\leq n$ varies. Since the divisor $k\Gamma_i$ is big and nef, by Theorem \ref{BB} we get \[\HH^p(X,\OO(-k\Gamma_i))=0 \ \ \ \ \ , \forall p \neq \dim P_{k\Gamma_i}=n \ . \]
Moreover, by Serre duality (SD)
\[ \HH^n(X,\OO(-k\Gamma_i))\stackrel{SD}{=}\HH^0(X,\OO(K_X+k\Gamma_i))^\vee\stackrel{\eqref{canonical divisor}}{=}\HH^0(X,\OO((k-1)\Gamma_i-\Delta_i))^\vee \ . \]
The characters $m \in M_{_{T'}}\subset\mathbb Z^{n+1}$ giving the splitting of $\HH^0(X,\OO((k-1)\Gamma_i-\Delta_i))$ into irreducible one-dimensional torus representations are exactly the ones satisfying the inequalities
\begin{equation}\label{conditions on lattice points for H0} 
	\langle m, u_S\rangle \leq \begin{cases}
	k-1 & \text{if } i \in S\\
	-1 & \text{if } i \notin S
	\end{cases} 
	\end{equation}
for any non-empty proper subset $S\subset [n+1]$. In particular, by \eqref{rays} we get 
\[ \begin{cases}
	\langle m, u_{\{i,j_1,\ldots, j_r\}}\rangle \leq (k-1)-r \\
	\langle m, u_{\{j_1,\ldots, j_r\}}\rangle \leq -r
	\end{cases} \] 
while by \eqref{rays relations} it holds
\[ u_{[n+1]\setminus \{t_1,\ldots, t_s\}}=-u_{\{t_1,\ldots, t_s\}} \implies \begin{cases}
	1\leq \langle m, u_{\{i,j_1,\ldots, j_r\}}\rangle \\
	-k+1 \leq \langle m, u_{\{j_1,\ldots, j_r\}}\rangle
	\end{cases} \ , \]
 where each $j_l\neq i$. Since we consider only proper subsets, $r$ can be at most $n-1$. Consider $r=n-1$ and recall that $k\in\{1, \ldots, n\}$. Thus, since $k< n+1$, we have $1\leq \langle m,u_{\{i,j_1,\ldots, j_{n-1}\}}\rangle \leq k-n < 1$, leading to a contradiction. Therefore, 
\[ \HH^n(X,\OO(-k\Gamma_i))=\HH^0(X,\OO((k-1)\Gamma_i-\Delta_i))^\vee=0, \]
and from the common vanishings for the sheaves $[\OO(-k\Gamma_i)]$ and $[\OO(-\ell \Delta_j)]$, we conclude that 
\begin{equation}\label{chi(kGamma)}
	\chi_{_T}([\OO(-k\Gamma_i)])=\chi_{_T}([\OO(-k\Delta_j)])=0, \ \ \  \forall 1\leq k\leq n . 
	\end{equation}
\end{proof}
	
It remains to deal with the divisors $D_{ij}^{k\ell}=k \Gamma_i+\ell \Delta_j$ for $k,\ell\geq 1$. Since they are big and nef, by Theorem \ref{TKV}, we know that $\HH^p(X,\OO(-D_{ij}^{k\ell}))=0$ for any $p\neq n$, thus we only have to determine
\begin{align*} 
	\HH^n(X_n,\OO(-D_{ij}^{k\ell})) = \HH^0(X_n,\OO(K_X+D_{ij}^{k\ell}))^\vee = \HH^0(X_n,\OO( (k-1)\Gamma_i-\Delta_i+ \ell \Delta_j))^\vee \ .
	\end{align*}
At this point we have obtained that $\chi_{_T}(\cA_{n,i}^a\otimes \cB_{n,j}^{n-a} )$ only depend on global sections:
\begin{equation}\label{chi as global sections}
	\chi_{_T}(\cA_{n,i}^a\otimes \cB_{n,j}^{n-a} ) = 1+\sum_{k=1}^a\sum_{\ell=1}^{n-a}(-1)^{n+k+\ell}\binom{a}{k}\binom{n-a}{\ell}\HH^0(X_n, \OO((k-1)\Gamma_i-\Delta_i+\ell \Delta_j))^\vee \ .
	\end{equation}
We simplify the notation by denoting the divisor $S_{ij}^{k\ell} := (k-1)\Gamma_i -\Delta_i + \ell \Delta_j$, which can be rewritten with respect to the prime divisors as follows:
\[ S_{ij}^{k\ell} \ = \ (k-1)\!\!\sum_{\{i,j\} \subset S}\!\!D_S \ + \ (k+\ell -1)\!\!\!\sum_{i \in S, \ j\notin S}\!\!D_S \ -\!\!\!\sum_{i \notin S, \ j \in S}\!\!D_S \ + \ (\ell - 1)\!\!\sum_{i,j \notin S}D_S \ . \]
The above description allows determining the system of inequalities that a character $m \in M_{_{T'}}\subset \mathbb Z^{n+1}$ has to satisfy so that $\HH^0(X_n,\OO(S_{ij}^{k\ell}))_m\neq 0$. Let $m_S=\sum_{h\in S}m_h=\langle m, u_S\rangle$, then we get
\begin{align}\label{system of m_S}
		1 & \leq m_{\{i,s_1,\ldots , s_r\}} \leq k+\ell -1 \nonumber\\
		-(\ell - 1) & \leq m_{\{i,j, s_1,\ldots, s_r\}} \leq k-1  \\
		-(k+\ell -1) & \leq m_{\{j,s_1,\ldots, s_r\}} \leq -1 \nonumber\\
		-(k-1) & \leq m_{\{s_1,\ldots, s_r\}} \leq \ell -1 \ . \nonumber
		\end{align}
for any subset $\{s_1,\ldots, s_r\}\subset [n+1]\setminus \{i,j\}$ (the empty subset is admitted too).

\begin{remark}
	For $i=j$, one gets $S_{ii}^{k\ell}=(k-1)\Gamma_i+(\ell-1)\Delta_i= \ (k-1)\!\!\sum_{i \in S}\!\!D_S \ + \ (\ell - 1)\!\!\sum_{i \notin S}D_S$, leading to the system of inequalities:
	\begin{align*}
		-(\ell -1) & \leq m_{\{i,s_1,\ldots , s_r\}} \leq k -1 \\
		-(k-1) & \leq m_{\{s_1,\ldots, s_r\}} \leq \ell -1, 
		\end{align*}
	which always admits $(0,\ldots,0)\in \mathbb Z^{n+1}$ as a solution. On the other hand, for $i\neq j$, the character $(0,\ldots, 0)\in \mathbb Z^{n+1}$ is never a solution of \eqref{system of m_S} for any $k,\ell \geq 1$, thus 
	\[\HH^0(X_n,\OO(S_{ij}^{k\ell}))_{(0,\ldots, 0)}=\begin{cases}
		1 & \text{if } i=j\\
		0 & \text{if } i\neq j
		\end{cases} \ . \]
	\end{remark}
	
\begin{remark}\label{H^0(S^{11})}
		For $i\neq j$ and $k=\ell=1$ the only solution is $m=e_i-e_j\in M_{_{T'}}$, that is
		\[ \HH^0(X_n,\OO(S_{ij}^{11}))= \HH^0(X,\OO(-\Delta_i+\Delta_j))= t_it_j^{-1}. \]
		\end{remark}

\begin{remark}\label{nested solutions}
	Given $1 \leq k\leq k' \leq a$ and $1 \leq \ell \leq \ell ' \leq n-a$, from the lower and upper bounds in the system \eqref{system of m_S}, we deduce that the lattice solutions for $(k,\ell)$ are lattice solutions for $(k',\ell ')$ too. Indeed, the coordinate intervals for solutions increase as $k$ and $\ell$ increase. It follows that
	\[ \HH^0(X_n,\OO(S_{ij}^{k\ell})) \subset 	\HH^0(X_n,\OO(S_{ij}^{k'\ell '})), \ \forall (k',\ell ')\geq (k,\ell),\]
thus for any character $m\in M_{_{T'}}$ such that $H^0(X_n,\OO(S_{ij}^{k\ell}))_m \neq 0$ one gets 
\[\HH^0(X_n,\OO(S_{ij}^{k\ell}))_m =	\HH^0(X_n,\OO(S_{ij}^{k'\ell '}))_m \ . \]
For instance, if $i\neq j$, Remark \ref{H^0(S^{11})} implies that $\HH^0(X_n,\OO(S_{ij}^{k\ell}))_{e_i-e_j}^\vee=t_jt_i^{-1}$ for any $k,\ell \geq 1$, thus we can compute its coefficients in the equivariant Euler characteristic $\chi_{_T}(\cal A_{n,i}^a\otimes \cal B_{n,j}^{n-a})$:
\[
\chi_{_T}(\cA_{n,i}^{a}\otimes \cB_{n,j}^{n-a} )_{e_j-e_i}  = \sum_{k=1}^a(-1)^{n+k}\binom{a}{k} \left[ \sum_{\ell=1}^{n-a}(-1)^\ell \binom{n-a}{\ell} \right] \left(t_jt_i^{-1}\right)= (-1)^n \left( t_jt_i^{-1}\right).
\]
	\end{remark}
The above remark allows restricting the search of (possible) irreducible summands of the virtual representation $\chi_{_T}(\cA_{n,i}^a\otimes \cB_{n,j}^{n-a})$ among the ones of $\HH^0(X,\OO(S_{ij}^{a ,n-a}))$.\\
\hfill\break
\indent In the \textit{Macaulay2} file \texttt{EquivariantEC.m2}, the implementation of this approach computing the equivariant Euler characteristic $\chi_{_T}(\cal A_{n,i}^a\otimes \cal B_{n,j}^{n-a})$ comes  under the function \texttt{LPCFormula} which takes the input parameters $i,j,a,n$.

\begin{example}[$n=3$] \label{examp}
Consider the $3$-dimensional permutohedral variety $X_3$ and we denote $\{i,j,p,q\}=[4]$. Starting from \eqref{chi as global sections} and solving \eqref{system of m_S} one gets:
	\begin{align*}
	\chi_{_T}(\cA_{3,i}\otimes \cB_{3,i}^{2} ) & = 1 + \sum_{\ell=1}^2(-1)^{3+1+\ell}\binom{2}{\ell}\HH^0\left(X_3,\OO\big((\ell-1)\Delta_i\big)\right)^\vee\\
	& = t_it_j^{-1} + t_it_p^{-1} + t_it_q^{-1}
    \end{align*}
    \begin{align*} 
	\chi_{_T}(\cA_{3,i}\otimes \cB_{3,j}^{2} ) & = 1+ \sum_{\ell=1}^2(-1)^{3+1+\ell}\binom{2}{\ell}\HH^0\left(X_3,\OO\big(-\Delta_i+\ell \Delta_j\big)\right)^\vee\\
	& = 1 - t_i^{-1}t_j + t_i^{-2}t_j^2 + t_i^{-1}t_j^2t_p^{-1} + t_i^{-1}t_j^2t_q^{-1}\\
	\chi_{_T}(\cA_{3,i}^2\otimes \cB_{3,i} ) & = 1 + \sum_{k=1}^2(-1)^{3+1+k}\binom{2}{k}\HH^0\left(X_3,\OO\big((k-1)\Gamma_i\big)\right)^\vee \\
    & = t_jt_i^{-1} + t_pt_i^{-1} +t_qt_i^{-1}\\
	\chi_{_T}(\cA_{3,i}^2\otimes \cB_{3,j} ) & = 1 + \sum_{k=1}^2(-1)^{3+1+k}\binom{2}{k}\HH^0\left(X_3,\OO\big((k-1)\Gamma_i-\Delta_i+ \Delta_j\big)\right)^\vee\\
	& = 1 - t_i^{-1}t_j + t_i^{-2}t_j^2 + t_i^{-2}t_jt_p + t_i^{-2}t_jt_q.
	\end{align*}
Here, we exploit the functions in 
\texttt{EquivariantEC.m2} to compute $\chi_{_T}(\cA_{n,i}^a\otimes \cB_{n,j}^{n-a} )$ for the choice of parameters $(i,j,a,n)=(1,2,2,3)$. The idea is to describe the system \eqref{system of m_S} for $k,\ell$ fixed by a linear inequality $N\leq B_{k,\ell}$ and to find its integral solutions as lattice points lying in the polyhedron defined by the matrix $N$ and the vector $B_{k,\ell}$: we do so using the package \texttt{Polyhedra}. The matrix $N$ of binary coefficients parameterizing the variables $m_i$'s in the system of inequalities \eqref{system of m_S} is computed by the function \texttt{MatFun(n)} and then we use the function \texttt{BMat($\{$i,j$\}$,k,l,n)} for constructing the vector $B_{k,\ell}$ of the upper bounds in \eqref{system of m_S}. Built on that the function \texttt{LPCFormula} takes the sum of the solutions for running $k$ and $\ell$:\\
\hfill\break
    \texttt{i1: load"EquivariantEC.m2"\\
    i2:     EC=LPCFormula(\{1,2\},2,3)\\
            o2:  $\texttt{q}_{_{-2,1,1,0}} - \texttt{q}_{_{-1,1,0,0}} + \texttt{q}_{_{-2,2,0,0}} + \texttt{q}_{_{-2,1,0,1}} + \texttt{q}_{_{0,0,0,0}}$\\}
\hfill\break
The output \texttt{o2} is in concordance with the Euler characteristic computed above 
\[ \chi_{_T}\left( \cA_{3,1}^2\otimes \cB_{3,2} \right)= 1-t_2t_1^{-1}+t_2^2t_1^{-2}+t_2t_3t_1^{-2}+t_2t_4t_1^{-2}.\]
\end{example}

\subsection{Second approach: equivariant localization}\label{sec:eqloc}

\indent \indent In this section, we exploit techniques from equivariant geometry, in particular equivariant $ K$-theory. As the hallmark of this theory, we use the localization formula \ref{formula}. \\

In fact, using this formula to compute the equivariant Euler characteristic for the sheaf $[\OO(-D_{ij}^{k\ell})]$ one needs to compute the class $[\OO(-D_{ij}^{k\ell})]_\sigma$ for any $\sigma\in \mathfrak S_{n+1}$ which corresponds to the piece-wise linear functions on the cones $C_\sigma$ defined by the coefficients of $-D_{ij}^{k\ell}$ with respect to the prime $T$-invariant divisors. In particular, it holds
\begin{align}\label{coeff -D}
	-D_{ij}^{k\ell} = (-k-\ell)\cdot\!\!\!\sum_{S\cap \{i,j\}=\{i\}}\!\!\!D_S + 0 \cdot \!\!\!\sum_{S\cap \{i,j\}=\{j\}}\!\!\!D_S- \ k\cdot \!\!\sum_{S\supset \{i,j\}}\!\!\!D_S - \ \ell\cdot \!\!\!\sum_{S\subset [n+1]\setminus\{i,j\}}\!\!\!D_S \ .
	\end{align}
Recall that, via the bijection \eqref{bij S_n X^T}, any $\sigma \in \mathfrak S_{n+1}$ uniquely corresponds to the cone
\[ C_\sigma=Cone\left(u_{S_1},\ldots, u_{S_n} \right)\]
where $S_h=\{\sigma(1), \sigma(2), \ldots , \sigma(h)\}$ for any $1\leq h\leq n$. For simplicity, we identify the cones with their defining flags 
\[C_\sigma=(S_1\subsetneq \ldots \subsetneq S_n) \ . \]

\begin{lemma}\label{Dij}
	Given $-D_{ij}^{k\ell} \in \Div^T(X_n)$ as in \eqref{coeff -D}, the Grothendieck class $[\OO(-D_{ij}^{k\ell})]\in K^0_T(X_n)$ has fiber at any $T$-invariant point $\sigma \in \mathfrak S_{n+1}\simeq X_n^T$ given by
	\begin{equation}\label{fiber at sigma of O(-D)}
		[\OO(-D_{ij}^{k\ell})]_\sigma= t_{\sigma(1)}^{-\ell}t_i^{-k}t_j^\ell t_{\sigma(n+1)}^{k}  \ .
		\end{equation}
	\end{lemma}
\begin{proof}
The value of the piece-wise linear functions on the cones $C_\sigma$ depends on the intersections $S_h\cap\{i,j\}$ and we have the following cases:
\begin{enumerate}
	\item[(i)] $\sigma(1)=i$ and $\sigma(n+1)=j$;
	\item[(ii)] $\sigma(p)=i$ for some $1 \lneq p \leq n$, and  $\sigma(n+1)=j$;
	\item[(iii)] $\sigma(1)=j$ and $\sigma(n+1)=i$;
	\item[(iv)] $\sigma(p)=j$ for some $1 \lneq p\leq n$, and $\sigma(n+1)=i$;
	\item[(v)] $\sigma(1)=i$ and $\sigma(p)=j$ for some $1\lneq p\leq n$;
	\item[(vi)] $\sigma(1)=j$ and $\sigma(p)=i$ for some $1\lneq p\leq n$;
	\item[(vii)] $\sigma(p)=i$ and $\sigma(q)=j$ for some $1\lneq p<q \leq n$;
	\item[(viii)] $\sigma(p)=j$ and $\sigma(q)=i$ for some $1\lneq p<q \leq n$.
	\end{enumerate}
Let us consider case (i). The cone $C_\sigma$ corresponds to the flag
\[ S_1=\{i\} \subsetneq S_2=\{i,\sigma(2)\} \subsetneq \ldots \subsetneq S_n=\{i,\sigma(2), \ldots, \sigma(n)\} \ .\]
Then for any above $S_h$ it holds $S_h\cap \{i,j\}=\{i\}$, that is from \eqref{coeff -D} we get a linear function $\phi_\sigma$ which has value $-k-\ell$ on any ray generator from the given cone, i.e.~corresponding to one $S_h$. 
In particular, $-k-\ell 
=\phi_\sigma(u_{\{i\}})$, while for any $h\geq 2$ we get $-k-\ell= 
\phi_\sigma(u_{S_h})=\phi_\sigma(u_{\{i\}}+u_{\sigma(2)}+\ldots + u_{\sigma(h)})$ which implies that $\phi_\sigma(u_{\sigma(p)})=0$ for any $2\leq p \leq n$. Thus for any $\sigma$ satisfying (i),  one gets 
\[[\OO(-D_{ij}^{kl})]_\sigma=t_i^{-k-\ell}  t_j^{k+\ell}\ . \]
By repeating the above arguments for each one of the cases (i)-(viii), we obtain the desired result. 
\end{proof}
\indent By replacing \eqref{fiber at sigma of O(-D)} in the localization formula \eqref{localization formula}, we get
\begin{align}\label{localization formula for O(-D)}
	 \chi_{_T}([\OO(-D_{ij}^{k\ell})]) & = \sum_{\sigma\in \mathfrak S_{n+1}} \dfrac{[\OO(-D_{ij}^{k\ell})]_\sigma}{\prod_{s=1}^{n}\left(1-t_{\sigma(s)}^{-1}t_{\sigma(s+1)}\right)} \\
  & = t_i^{-k}t_j^\ell \sum_{\sigma\in \mathfrak S_{n+1}} \dfrac{t_{\sigma(1)}^{-\ell} t_{\sigma(n+1)}^{k}}{\prod_{s=1}^{n}\left(1-t_{\sigma(s)}^{-1}t_{\sigma(s+1)}\right)} \ .
\end{align}

\noindent By using the vanishings in the previous section, we get
{\small \begin{align}\label{chi(F)withSigma}
	\chi_{_T}(\cA_{n,i}^a\otimes \cB_{n,j}^{n-a} ) & = 1 + \sum_{k=1}^{a}\sum_{\ell=1}^{n-a}(-1)^{k+\ell}\binom{a}{k}\binom{n-a}{\ell}\chi_{_T}\left( [\OO(-D_{ij}^{k\ell})]\right) \\
	&= 1+\sum_{k=1}^{a}\sum_{\ell=1}^{n-a}(-1)^{k+\ell}\binom{a}{k}\binom{n-a}{\ell}\sum_{\sigma\in \mathfrak S_{n+1}} \dfrac{t_{\sigma(1)}^{-\ell}t_i^{-k}t_j^\ell t_{\sigma(n+1)}^{k}}{\prod_{s=1}^{n}\left(1-t_{\sigma(s)}^{-1}t_{\sigma(s+1)}\right)} \ .
\end{align}}
\begin{remark}
    Working with the torus $T'$, we may eliminate $t_{n+1}$ taking into account that the sum of coordinates of each character is zero. 
    \end{remark}

\begin{example}\label{examp:second approach}($n=3$)
    We describe the implementation of our code for the second approach appearing in the file \texttt{EquivariantEC.m2}. As in Example \ref{examp}, we use it for computing the equivariant Euler characteristic $\chi_{_T}(\cA_{3,1}^2\otimes \cB_{3,2} )$. After constructing the list $S$ of permutations on $n+1$ elements via the function \texttt{sigma(n)},  the function \texttt{sigmaterms(S,\{i,j\},\{k,l\},n)} computes the equivariant Euler characteristic $\chi_{_T}([\OO(-D_{ij}^{k\ell})])$ in \eqref{localization formula for O(-D)} via the localization formula, and finally, the function \texttt{ECIJ(IJ,a,n)} takes the sum over all $k,\ell$ in order to obtain $\chi_{_T}(\cA_{n,i}^a\otimes \cB_{n,j}^{n-a})$.\\
    \hfill\break
    \texttt{i1:  load"EquivariantEC.m2" \vspace{1mm} \\
    i2 : EC=ECIJ(\{1,2\},2,3) \vspace{1mm} \\
    o2 = \space $\dfrac{\texttt{t}_1^2-\texttt{t}_1\texttt{t}_2+\texttt{t}_2^2 + \texttt{t}_2\texttt{t}_3+\texttt{t}_2\texttt{t}_4}{\texttt{t}_1^2}$ \\}
    \hfill\break
    The above output agrees with the output \texttt{o2} in Example \ref{examp}.
    \end{example}
    
\begin{remark}[General linearization of $\Gamma$ and $\Delta$]
The above arguments can be reformulated in terms of general divisors $\Gamma$ and $\Delta$ on $X_n$ without any choice of coordinates $i,j$. We recall the notation $T=(\mathbb C^\times)^{n+1}$ and $T'= T/\langle \mathbbm{1}\rangle$ for the big and small torus respectively, and $M_T$ and $M_{T'}$ for their corresponding character lattices. Consider the line bundle $\OO(1)$ be the line bundle on $\PP^n$ having fiber $\big[\OO(1)\big]_{[e_r]}=t_r^{-1}\in M_T$ at the $T$-fixed point $[e_r]\in \mathbb P^n$ corresponding to the $r$-th coordinate point. If one denotes by $\OO(1)_i$ the line bundle corresponding to the hyperplane $H_i\subset \mathbb P^n$, then one has the equivalence of $K$-classes $[\OO(1)_i]=t_i[\OO(1)]$ so that the fiber at $[e_r]\in \mathbb P^n$ is $\big[ \OO(1)_i\big]_{[e_r]}=t_r^{-1}t_i\in M_{T'}$. In particular, 
\[ [\OO(-\Gamma_i)]=t_i^{-1}[\OO(-\Gamma)], \qquad [\OO(-\Delta_j)]\stackrel{\clubsuit}{=}t_j[\OO(-\Delta)]\]
where the equality ($\clubsuit$) follows from the action of the Cremona map. In light of this, the $K_{_{T'}}$-class $[\OO(-D_{ij}^{k\ell})]=[\OO(-k\Gamma_i-\ell \Delta_j)]$ can be rewritten as
\[[\OO(-D_{ij}^{k\ell})]=[\OO(-\Gamma_i)]^k[\OO(-\Delta_j)]^\ell=\big( t_i^{-1}[\OO(-\Gamma)]\big)^k\big(t_j[\OO(-\Delta)]\big)^\ell=t_i^{-k}t_j^{\ell}[\OO(-k\Gamma-\ell \Delta)]\]
and from \eqref{chi(F)withSigma} one gets straightforward
\[ \chi_{_{T'}}\left( \big([\cO]-[\cO(-\Gamma_i)]\big)^a\big([\cO]-[\cO(-\Delta_j)]\big)^b \right)=\chi_{_T}\left(\big([\cO]-t_i^{-1}[\cO(-\Gamma)]\big)^a\big([\cO]-t_j[\cO(-\Delta)]\big)^b \right) \ . \]
If we denote by $\cA_{n}=[\OO]-[\OO(-\Gamma)]$ and $\cB_n=[\OO]-[\OO(-\Delta)]$, then we deduce the following formula, similar to \eqref{chi(F)withSigma} but independent from $i,j$
\begin{align*} 
    \chi_{_T}\left(\cA_n^a\otimes \cB_n^{n-a}\right) & = 1 + \sum_{k=1}^a\sum_{\ell=1}^{n-a}(-1)^{k+\ell}\binom{a}{k}\binom{n-a}{\ell}\chi_{_T}\left( [\OO(-k\Gamma-\ell \Delta)] \right) \\
    & = 1 + \sum_{k=1}^a\sum_{\ell=1}^{n-a}(-1)^{k+\ell}\binom{a}{k}\binom{n-a}{\ell}\sum_{\sigma\in \mathfrak S_{n+1}} \dfrac{t_{\sigma(1)}^{-\ell} t_{\sigma(n+1)}^{k}}{\prod_{s=1}^{n}\left(1-t_{\sigma(s)}^{-1}t_{\sigma(s+1)}\right)} \ .
    \end{align*}
    In the file \texttt{equivariantEC.m2} we have implemented the function \texttt{EquiChar(n,a)} for the computation of $\chi_{_T}(\cA_n^a\otimes \cB_n^{n-a})$ too: notice that this function differs from the function \texttt{ECIJ(\{i,j\},a,n)} only in the definition of the numerators of the summands in the above expression. For instance, for computing $\chi_{_T}(\cA_3^2\otimes \cB_3)$ one has the following input and output: \\
    \hfill\break
    \texttt{i1:  load"EquivariantEC.m2"\\
    i2 : EquiChar(3,2)\\
    o2 = \space $\texttt{t}_1+\texttt{t}_2+\texttt{t}_3+\texttt{t}_4-1$\\}
\end{remark}
\section{Recursive formula for Euler characteristic}\label{Recursive}
\indent \indent In this section we present a recursive formula for computing the equivariant Euler characteristic $\chi_{_T}(\cA_{n,i}^a\otimes \cB_{n,j}^{n-a})$. The key will be a restriction to one of the divisors $\Gamma$ or $\Delta$, which can be represented as the sum of torus-invariant divisors $D_S$, so that $D_S\cong X_{|S|-1}\times X_{|E\setminus S|-1}$. Hence, the restriction to each such a divisor, and moreover to their intersections, comes with two natural projections to lower dimensional permutohedral varieties. We use this to write the corresponding Euler characteristic as the product of the Euler characteristic of simpler line bundles on the lower dimensional permutohedral varieties.\\
\indent More precisely, let $D\subset X_n$ be an irreducible torus-invariant divisor embedded with $i:D\hookrightarrow X_n$, and let $\pi:X_n\longrightarrow \lbrace pt\rbrace$ be the projection to a point. As $X_n$ is smooth, so is $D$. We recall that the Euler characteristic map 
\[\chi_{_{X_n}}: K_T(X_n)\longrightarrow K_T(pt)\] 
is defined as the pushforward map along the projection $\pi$. Therefore, for a sheaf $\mathcal{F}$ on $X_n$, we obtain
\begin{align*}
    \chi_{_{X_n}}(([\cO]-[\cO(-D)])\otimes [\mathcal{F}])&=[\pi_*(([\cO]-[\cO(-D)])\otimes [\mathcal{F}])]\\
    &=[\pi_*([i_*\cO_D]\otimes [\mathcal{F}])]\\
    &=[\pi_*i_*([\cO_D]\otimes i^*[\mathcal{F}])]\\
     &=\chi_{_{D}}(i^*[\mathcal{F}])
\end{align*}
where the second equality follows from the exact sequence 
\[0\rightarrow\cO(-D)\rightarrow\cO \rightarrow i_*\cO_D \rightarrow 0,\]
and the third equality follows from the projection formula. We conclude that:
\begin{align}\label{eq:l=1}
     \chi_{_{X_n}}(([\cO]-[\cO(-D)])\otimes [\mathcal{F}])=\chi_{_{D}}([\mathcal{F}_{|D}]).
\end{align}
This is the key to allowing the restriction to a divisor $D$. 
\begin{theorem}
Let $D=\sum_{i=1}^\ell D_i$ be a divisor such that the divisors $D_i$ are prime, they intersect transversally and each non-empty intersection of them is smooth.
Then the following recursive formula holds:
\begin{align}\label{main}
   \chi_{_{X_n}}\bigg(([\cO]-[\cO(-D)])\otimes [\mathcal{F}]\bigg) = \sum_{\emptyset\neq \{j_1,\ldots, j_r\}\subset [\ell]}(-1)^{r+1}\chi_{_{D_{j_1}\cap \ldots \cap D_{j_r}}}\bigg(\bigg[\mathcal F_{|_{D_{j_1}\cap \ldots \cap D_{j_r}}}\bigg]\bigg) 
\end{align} 
\noindent where the sum is taken over all possible nonempty subsets of $[\ell]$.
\end{theorem}
\begin{proof}
We proceed by induction on $\ell$, the case $\ell=1$ being \eqref{eq:l=1}. Let $D'=\sum_{i=1}^{\ell-1} D_i$. We have
\begin{align*}
   &\chi\left( ([\cO]-[\cO(-D')])\otimes ([\cO]-[\cO(-D_\ell)])\otimes [\mathcal{F}]\right)\\
   &=\chi\left( ([\cO]-[\cO(-D')]-[\cO(-D_\ell)]+[\cO(-D'-D_\ell)])\otimes [\mathcal{F}]\right)\\
   &=\chi\left( (([\cO]-[\cO(-D')])+([\cO]-[\cO(-D_\ell)])-([\cO]-[\cO(-D'-D_\ell)]))\otimes [\mathcal{F}]\right).
\end{align*}
As $D=D'+D_\ell$ and the Euler characteristic is additive, we obtain
\begin{align*}
    \chi\left(([\cO]-[\cO(-D)])\otimes [\mathcal{F}]\right)&=\chi\left(([\cO]-[\cO(-D')])\otimes [\mathcal{F}])+\chi(([\cO]-[\cO(-D_\ell)])\otimes [\mathcal{F}]\right)\\
    & \ \ \ -\chi\left(([\cO]-[\cO(-D')])\otimes ([\cO]-[\cO(-D_\ell)])\otimes [\mathcal{F}]\right).
\end{align*}
Applying induction to the first term we obtain the part of the RHS of \eqref{main} where $\{j_1,\dots,j_r\}\subset [\ell-1]$. The second term contributes one summand $\{j_1\}=\{\ell\}$. For the last summand
\[\chi\left(([\cO]-[\cO(-D')])\otimes ([\cO]-[\cO(-D_\ell)])\otimes [\mathcal{F}]\right)\]
we again apply induction, treating $([\cO]-[\cO(-D_\ell)])\otimes [\mathcal{F}]$ as one element. We conclude by noting that:
\begin{align*}
\chi_{_{D_{j_1}\cap \ldots \cap D_{j_r}}}\left(\left[ \big(([\cO]-[\cO(-D_\ell)])\otimes [\mathcal{F}]\big)_{|_{D_{j_1}\cap \ldots \cap D_{j_r}}}\right]\right)=&
\chi_{_{D_{j_1}\cap \ldots \cap D_{j_r}\cap D_\ell}}\left(\left[\mathcal{F}_{|_{D_{j_1}\cap \ldots \cap D_{j_r}\cap D_\ell}}\right]\right),
\end{align*}
by the case $\ell=1$. Thus the last term corresponds to $r$-tuples $j_1,\dots,j_r$ where $r>1$ and $j_r=\ell$. This concludes the inductive step and hence the proof.
\end{proof}
 For a flag $S_1\subsetneq S_2\subsetneq \ldots \subsetneq S_k$ of proper non-empty subsets of $E$, consider the corresponding intersection:
$$D_{S_1,\ldots,S_k}:=D_{S_1}\cap \cdots \cap D_{S_k}. $$
\begin{corollary}\label{resTOGamma}
By restriction to $\Gamma_i$, for $a>0$, it holds
\begin{align*} 
    \chi_{_{X_n}}(\cA_{n,i}^a\otimes \cB_{n,j}^{n-a} )& =\chi_{_{X_n}}(\cA_{n,i}\otimes\cA_{n,i}^{a-1}\otimes \cB_{n,j}^{n-a})\\
    & = \sum_{k=1}^{n}\sum_{i \in S_1\subsetneq \ldots \subsetneq S_k}(-1)^{k+1}\chi_{_{D_{S_1,\ldots, S_k}}}\left({\cA_{n,i}^{a-1}}_{|_{D_{S_1,\ldots,S_k}}}\otimes {\cB_{n,j}^{n-a}}_{|_{D_{S_1,\ldots,S_k}}}\right)
    \end{align*}
    \noindent where, after fixing $k$, the second sum is taken over all possible flags of length $k$ whose first subset contains $i$. A similar formula holds for restriction to $\Delta_j$, where the second sum runs over all possible flags not containing $j$.
\end{corollary}
\begin{theorem}\label{thm:rec}
By restriction to $\Gamma_i$, for $a>0$, it holds
{\small \begin{align*}\label{eq:fin}
     \chi_{_{X_n}}(\cA_{n,i}^a\otimes \cB_{n,j}^{n-a} )&=\sum_{i \in S\subsetneq E}\chi\left(\cA_{_{|E\setminus S|-1,i}}^{a-1} \right)\chi\left(\cB_{_{|S|-1,j}}^{n-a} \right)\\
     & \ \ + \sum_{k=2}^{n}\sum_{i \in S_1\subsetneq S_k\subsetneq E}(-1)^{k+1}(k-1)!\stirlingii{|S_k\setminus S_1|}{k-1}\chi\left(\cA_{_{|E\setminus S_{k}|-1,i}}^{a-1} \right)\chi\left(\cB_{_{|S_1|-1,j}}^{n-a} \right)\\
    & = \sum_{i \in S\subsetneq E}\chi\left(\cA_{_{|E\setminus S|-1,i}}^{a-1} \right)\chi\left(\cB_{_{|S|-1,j}}^{n-a} \right) + \sum_{i \in S\subsetneq S'\subsetneq E}(-1)^{|S'\setminus S|}\chi\left(\cA_{_{|E\setminus S'|-1,i}}^{a-1} \right)\chi\left(\cB_{_{|S|-1,j}}^{n-a} \right)
\end{align*} }
where $\stirlingii{|S_k\setminus S_1|}{k-1}$ is the Stirling function counting the number of $k-1$ partitions of the set $S_k\setminus S_1$ into non-empty subsets, and we set $\stirlingii{*}{0}=1$. A similar formula holds for restriction to $\Delta_j$, where the second sum runs over all possible $2$-flags not containing $j$.
\end{theorem}
\begin{proof}
    We note that each intersection can be described as a product of lower dimensional permutohedral varieties
$$ D_{S_1,\ldots, S_k}\cong X_{|S_1|-1}\times X_{|S_2\setminus S_1|-1}\times \ldots \times X_{|S_k\setminus S_{k-1}|-1}\times X_{|E\setminus S_k|-1}$$
where the torus $(\CC^\times)^{|S_i\setminus S_{i-1}|}$ with coordinates $(t_\ell)_{\ell\in S_i\setminus S_{i-1}}$ acts on the $i$-th factor. Thus, one has two natural projections onto lower-dimensional permutohedral varieties
\begin{center}
	\begin{tikzpicture}[scale=2.5]
	\node(D) at (0,1.2){$D_{S_1\ldots S_k}$};
	\node(X) at (0,0.6){$X_{|S_1|-1}\times X_{|E\setminus S_k|-1}$};
	 \node(PP) at (-0.6,0){$X_{|S_1|-1}$};
	\node(PP^q) at (0.6,0){$X_{|E\setminus S_k|-1}$};
\path[font=\scriptsize,>= angle 90]
	(PP)  (PPvee)
	(D) edge [->] node [left] {$\phi$} (X)
	(X) edge [->] node [left] {$\pi_1$} (PP)
	(X) edge [->] node [right] {$\pi_2$} (PPvee);	
	\end{tikzpicture}
\end{center}
so that 
\begin{align*}
    {\cA_{n,i}}_{|_{D_{S_1\ldots S_k}}}\cong \phi^*\circ \pi_2^*(\cA_{_{|E\setminus S_k|-1,i}}), \ \ \ \ 
     {\cB_{n,j}}_{|_{D_{S_1\ldots S_k}}} \cong\phi^*\circ \pi_1^*(\cB_{_{|S_1|-1,j}})
\end{align*}
Therefore, we obtain
\begin{align*}
 \chi_{_{D_{S_1\ldots S_k}}}\left({\cA_{n,i}^{a-1}}_{|_{D_{S_1\ldots S_k}}}\otimes {\cB_{n,j}^{n-a}}_{|_{D_{S_1\ldots S_k}}}\right)&=\chi_{_{X_{|S_1|-1}\times X_{|E\setminus S_k|-1}}}\left(\pi_2^*(\cA^{a-1}_{_{|E\setminus S_k|-1,i}}) \otimes \pi_1^*(\cB^{n-a}_{_{|S_1|-1,j}})\right) \\
 &=\chi_{_{X_{|E\setminus S_k|-1}}}\left(\cA^{a-1}_{_{|E\setminus S_k|-1,j}}\right)\chi_{_{X_{|S_1|-1}}}\left(\cB^{n-a} _{_{|S_1|-1,j}}\right),
\end{align*}
 where the second equality follows from the K\"{u}nneth formula, and the first one follows from the definition of the Euler characteristic, as pushforward map along the projection to a point. By Corollary \ref{resTOGamma}, we obtain
 \begin{align*}
      \chi_{_{X_n}}(\cA_{n,i}^a\otimes \cB_{n,j}^{n-a} ) &= \sum_{k=1}^{n}\sum_{i \in S_1\subsetneq \ldots \subsetneq S_k}(-1)^{k+1}\chi\left(\cA_{_{|E\setminus S_{k}|-1,i}}^{a-1} \right)\chi\left(\cB_{_{|S_1|-1,j}}^{n-a} \right)
 \end{align*}
 where, after fixing $k$, the second sum is taken over all possible flags of length $k$ whose first subset contains $i$. Since only the tail and the head of a flag are seen by the formula, we can reduce to length-$2$ flags $S_1\subsetneq S_k$ and use the Stirling function $\stirlingii{|S_k\setminus S_1|}{k-1}$, multiplied with $(k-1)!$, to count the number of $k$-flags with such a given head and tail. 

 To obtain the final formula we need to prove that for any $s>0$:
 \[\sum_{k=1}^{s} (-1)^{k} k! \stirlingii{s}{k}=(-1)^{s}.\]

 The first proof of the above formula is by induction on $s$, the case $s=1$ being trivial. For $s>1$, we note that $k! \stirlingii{s}{k}$ counts the number of ordered partitions $(A_1,\dots, A_k)$ of the set \{1,\dots,s\}. All partitions for which $A_1\neq \{s\}$ can be paired as follows. If $A_l=\{s\}$ then we pair $(A_1,\dots, A_k)$ with $(A_1,\dots, A_{l-2}, A_{l-1}\cup \{s\}, A_{l+1},\dots, A_k)$. Similarly, if each $A_l\neq \{s\}$ then we must have $s\in A_l$ for some $l$. We pair the partition by splitting $A_l$ into $A_l\setminus \{s\}$, followed by $\{s\}$. Each pair contributes exactly one $1$ and one $-1$ to the sum. Thus it is enough to compute the contributions when $A_1=\{s\}$, which we know by induction. 

 For the second proof we recall the well-known formula:
 \[\sum_{k=0}^{s} \stirlingii{s}{k} x(x-1)\cdots (x-(k-1))=(x)^s.\]
 Our formula follows by substituting $x=-1$.
\end{proof}
    We remark  that in the computation of the terms $\chi(\cA_{n,i}^{a-1})$ and $\chi(\cB_{n,j}^{n-a})$ in Theorem \ref{thm:rec}, one can restrict further, respectively to $\Gamma_i$ and $\Delta_j$, so that the computation boils down to the computation of the Euler characteristic over a point.\\

    Julian Weigert implemented our recursive formula following the described procedure:
    \begin{enumerate}
        \item Precompute the Euler characteristic $\chi_{_{X_n}}(\cA_{n}^a)$ for various $n$ and $a$. This is achieved by using the recursive approach described above for computing $\chi_{_{X_n}}(\cA_{n,i}\otimes \cA_{n}^{a-1})$.
        \item From the previous point one easily obtains $\chi_{_{X_n}}(\cA_{n,i}^a)$ and $\chi_{_{X_n}}(\cB_{n,j}^{b} )$ by substitution.
        \item The final formula is obtained as the last one in Theorem \ref{thm:rec}.
    \end{enumerate}

We may reconstruct the formula for the usual Euler characteristic from our formulas for the equivariant Euler characteristic. We recall that the reduced chromatic polynomial of the uniform matroid on $E$ of rank $r$ equals:
\[\sum_{i=1}^{r} (-1)^{r-i} \binom{n}{r-i}t^i.\]
This is equivalent to the following corollary.
\begin{corollary}
In the non-equivariant setting:
\[\chi_{_{X_n}}(\cA_{n,i}^a\otimes \cB_{n,j}^{n-a} )=\binom{n}{a}.\]
\end{corollary}
\begin{proof}
    We note that intersecting more than dimension of the variety many divisors always return zero in the non-equivariant setting. In particular, the only non-zero contributions in Theorem \ref{thm:rec} may come from summands where both $a-1\leq |E\setminus S_k|-1$ and $n-a\leq |S_1|-1$. This is only possible when $S_1=S_k$, i.e.~$k=1$ and $|S_1|=n-a+1$. Applying further the induction, we see that each such contribution equals one. Thus we have to compute the number of subsets of $E$ of cardinality $n-a+1$ that contain a fixed index $i$. Clearly, there are $\binom{n}{n-a}=\binom{n}{a}$ such subsets.
\end{proof}

We recall that, in order to determine the Euler characteristic of a sheaf, one needs to determine its class as a polynomial at each torus fixed point corresponding to a permutation $\sigma$. In particular, when restricting to $D_{S_1,\ldots, S_k}$, one needs to compute the class of the line bundle at the torus fixed points on $D_{S_1,\ldots, S_k}$, which correspond to the permutations in $\mathfrak S_{n+1}$ of the form $(\sigma_1,\sigma_{2},\ldots,\sigma_k,\sigma_{e-k})$ so that $\sigma_i$ is a permutation of the set $S_i\setminus S_{i-1}$, and $\sigma_{e-k}$ is a permutation of the elements in $E\setminus S_k$. In the following, we recompute our example in the recursive approach by explicitly exhibiting the contribution of the intersection terms.

\begin{example}
We compute $\chi(\cA_{3,1}^2\otimes \cB_{3,2})$ on the $3$-dimensional permutohedral. In this case, considering the restriction to $\cA_{3,1}$, there are $7$ divisors, corresponding to proper non-empty subsets of $\lbrace 1,2,3,4\rbrace$ containing $1$, that is $D_1, D_{12}, D_{13}, D_{14}, D_{123}, D_{124}, D_{134}$ intersecting as demonstrated in the picture, and respectively colored in red, three blue, and three green.
\begin{figure}[h]
        \centering
        \includegraphics[width=3cm]{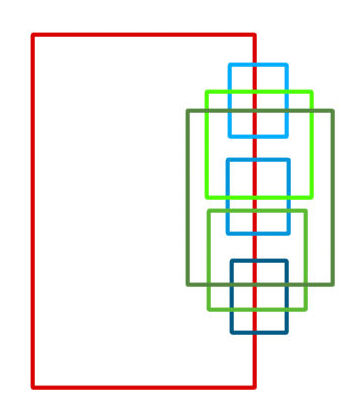}
        \label{div}
    \end{figure}
    \\
We first consider $D_{12}\cong X_{12}\times X_{34}$ with four torus fixed points corresponding to the permutations $1234, 1243, 2134, 2143$, where the torus $(\CC^\times)^2$ with coordinates $(t_1,t_2)$ and $(t_3,t_4)$ acts respectively on each factor. Using the recursive formula, and the fact that $\chi(\cA_{n,i})=\chi(\cB_{n,j})=1$ for $n\geq 1$, we have
$$\chi_{_{D_{12}}}(\cA_{3,1}\otimes \cB_{3,2} )=\chi(\cA_{1,1})\chi(\cB_{1,2}) = 1$$
Similarly, $\chi_{_{D_{13}}}(\cA_{3,1}\otimes \cB_{3,2} )=\chi_{_{D_{14}}}(\cA_{3,1}\otimes \cB_{3,2} )=1$.
Now, we consider the restriction to $D_1\cong X_1\times X_{234}$ which contains $6$ torus fixed points, and we have
$$
\chi_{_{D_1}}(\cA_{3,1}\otimes \cB_{3,2} )=\chi(\cA_{2,1})\chi(\cB_{0,2})= \chi(\cB_{0,2})=(1-t_2t_1^{-1})
$$
Now, considering one of the intersections, $D_1\cap D_{12}\cong X_1\times X_2\times X_{34}$, we obtain
$$\chi_{_{D_1\cap D_{12}}}(\cA_{3,1}\otimes \cB_{3,2})= \chi(\cA_{1,1})\chi(\cB_{0,2})= \chi(\cB_{0,2})=(1-t_2t_1^{-1}).$$
With the same argument for the remaining (intersection) terms, we have
\begin{align*}
 &\chi_{_{D_{123}}}(\cA_{3,1}\otimes \cB_{3,2} )=\chi(\cA_{0,1})\chi(\cB_{2,2})=(1-t_4t_1^{-1})\\
& \chi_{_{D_{134}}}(\cA_{3,1}\otimes \cB_{3,2} )=\chi(\cA_{0,1})\chi(\cB_{2,2})=(1-t_2t_1^{-1}) \\
& \chi_{_{D_{124}}}(\cA_{3,1}\otimes \cB_{3,2} )=\chi(\cA_{0,1})\chi(\cB_{2,2})=(1-t_3t_1^{-1}) \\
& \chi_{_{D_1\cap D_{13}}}(\cA_{3,1}\otimes \cB_{3,2})= \chi(\cA_{1,1})\chi(\cB_{0,2})=(1-t_2t_1^{-1})\\
& \chi_{_{D_1\cap D_{14}}}(\cA_{3,1}\otimes \cB_{3,2})= \chi(\cA_{1,1})\chi(\cB_{0,2})=(1-t_2t_1^{-1})\\
& \chi_{_{D_{12}\cap D_{123}}}(\cA_{3,1}\otimes \cB_{3,2})= \chi(\cA_{0,1})\chi(\cB_{1,2})=(1-t_4t_1^{-1})\\
 &\chi_{_{D_{12}\cap D_{124}}}(\cA_{3,1}\otimes \cB_{3,2})= \chi(\cA_{0,1})\chi(\cB_{1,2})=(1-t_3t_1^{-1})\\
& \chi_{_{D_{13}\cap D_{123}}}(\cA_{3,1}\otimes \cB_{3,2})= \chi(\cA_{0,1})\chi(\cB_{1,2})=(1-t_4t_1^{-1})\\
 &\chi_{_{D_{13}\cap D_{134}}}(\cA_{3,1}\otimes \cB_{3,2})= \chi(\cA_{0,1})\chi(\cB_{1,2})=(1-t_2t_1^{-1})\\
& \chi_{_{D_{14}\cap D_{134}}}(\cA_{3,1}\otimes \cB_{3,2})= \chi(\cA_{0,1})\chi(\cB_{1,2})=(1-t_2t_1^{-1})\\
& \chi_{_{D_{14}\cap D_{124}}}(\cA_{3,1}\otimes \cB_{3,2})= \chi(\cA_{0,1})\chi(\cB_{1,2})=(1-t_3t_1^{-1})\\
& \chi_{_{D_1\cap D_{123}}}(\cA_{3,1}\otimes \cB_{3,2})= \chi(\cA_{0,1})\chi(\cB_{0,2})=(1-t_4t_1^{-1})(1-t_2t_1^{-1})\\
 & \chi_{_{D_1\cap D_{124}}}(\cA_{3,1}\otimes \cB_{3,2})= \chi(\cA_{0,1})\chi(\cB_{0,2})=(1-t_3t_1^{-1})(1-t_2t_1^{-1})\\
  & \chi_{_{D_1\cap D_{134}}}(\cA_{3,1}\otimes \cB_{3,2})= \chi(\cA_{0,1})\chi(\cB_{0,2})=(1-t_2t_1^{-1})(1-t_2t_1^{-1})
  \end{align*}
  \begin{align*}
   & \chi_{_{D_1\cap D_{12}\cap D_{123}}}(\cA_{3,1}\otimes \cB_{3,2})= \chi(\cA_{0,1})\chi(\cB_{0,2})=(1-t_4t_1^{-1})(1-t_2t_1^{-1})\\
   & \chi_{_{D_1\cap D_{12}\cap D_{124}}}(\cA_{3,1}\otimes \cB_{3,2})= \chi(\cA_{0,1})\chi(\cB_{0,2})=(1-t_3t_1^{-1})(1-t_2t_1^{-1})\\
    & \chi_{_{D_1\cap D_{13}\cap D_{123}}}(\cA_{3,1}\otimes \cB_{3,2})= \chi(\cA_{0,1})\chi(\cB_{0,2})=(1-t_4t_1^{-1})(1-t_2t_1^{-1})\\
    & \chi_{_{D_1\cap D_{13}\cap D_{134}}}(\cA_{3,1}\otimes \cB_{3,2})= \chi(\cA_{0,1})\chi(\cB_{0,2})=(1-t_2t_1^{-1})(1-t_2t_1^{-1})\\
    & \chi_{_{D_1\cap D_{14}\cap D_{124}}}(\cA_{3,1}\otimes \cB_{3,2})= \chi(\cA_{0,1})\chi(\cB_{0,2})=(1-t_3t_1^{-1})(1-t_2t_1^{-1})\\
      & \chi_{_{D_1\cap D_{14}\cap D_{134}}}(\cA_{3,1}\otimes \cB_{3,2})= \chi(\cA_{0,1})\chi(\cB_{0,2})=(1-t_2t_1^{-1})(1-t_2t_1^{-1})
\end{align*}
Now, considering the alternating sum, we obtain 
\[\chi(\cA_{3,1}^2\otimes \cB_{3,2})=1-t_2t_1^{-1}+t_2^2t_1^{-2}+t_2t_3t_1^{-2}+t_2t_4t_1^{-2} \ , \]
as in Example \ref{examp}. In  \texttt{EquivariantEC.m2}, the function \texttt{eulerCharacteristic(3,2,1,2)} computes this Euler characterstic.
\end{example}

\end{document}